\documentclass[10pt]{article}

\usepackage{amsmath,amsfonts,amssymb,amsthm,mathrsfs}
\usepackage{graphicx}
\usepackage{stmaryrd}
\usepackage{hyperref}
\usepackage{tikz}

\usepackage[left=3.5cm,right=3.5cm,top=3.5cm,bottom=3.5cm]{geometry}

\theoremstyle{definition}
\newtheorem{definition}{Definition}
\newtheorem{remark}[definition]{Remark}
\newtheorem{example}[definition]{Example}

\theoremstyle{mytheorem}
\newtheorem{theorem}[definition]{Theorem}
\newtheorem{lemma}[definition]{Lemma}

\newtheorem{assumption}[definition]{Assumption}

 \newcommand{\W}{\mathscr{W}}

\newcommand{\equlaw}{\stackrel{(d)}{=}}

\renewcommand{\P}{\mathbf{P}} 
\newcommand{\E}{\mathbf{E}}

\renewcommand{\S}{\mathcal{S}}
\newcommand{\Per}{{\rm Per}}

\newcommand{\Lip}{{\rm Lip}}

\newcommand{\Vol }{{\rm Vol}}

\newcommand{\corners}{{\rm corners}}
 
\renewcommand{\v}{\mathbf{v}}
\renewcommand{\u}{\mathbf{u}}

\newcommand{\N}{\mathscr{N}}

\newcommand{\p}{\mathsf{P}}

\newcommand{\n}{\mathbf{n}}

\newcommand{\C} {\mathcal{C} }

\newcommand{\EC}{Euler characteristic}
   
\title{Bicovariograms and \EC~of random fields excursions}

 \author{ {Rapha\"el  Lachi\`eze-Rey} \footnote{
 {raphael.lachieze-rey@parisdescartes.fr,Laboratoire MAP5, 45 Rue des Saints-P\`eres, 75006 Paris},
 {Universit\'e Paris Descartes, Sorbonne Paris Cit\'e}}}

\date{}

\begin{document}

 \maketitle

\begin{abstract}
Let $f$ be a $\mathcal{C}^{1}$  bivariate function  with Lipschitz derivatives, and $F=\{x\in \mathbb{R}^{2}:f(x)\geqslant  \lambda \}$ an upper level set of $f$, with $\lambda \in \mathbb{R}$. We present a new identity giving  the \EC~of $F$ in terms of its three-points indicator functions. A bound on the number of connected components of $F$ in terms of the values of $f$ and its gradient, valid in higher dimensions, is also derived. In dimension $2$, if $f$ is a random field, this bound allows to pass the former identity to expectations   if $f$'s partial derivatives have Lipschitz constants with finite moments of sufficiently high order, without requiring bounded conditional densities. This approach provides   an expression of the mean \EC~in terms of the field's  third order marginal. Sufficient conditions and explicit formulas are given for Gaussian fields, relaxing the usual $\mathcal{C}^{2}$ Morse hypothesis.\\

\end{abstract}


{\bf keywords:} { Random fields}, \EC, {Gaussian processes}, {covariograms},{ intrinsic volumes},{ $C^{1,1}$ functions
}

 \section{Introduction} 
 
 The geometry of   random fields excursion sets has been a subject of intense research over the  last two decades. Many authors are concerned with the computation of the mean \cite{AdlSam,AST,AdlTay03, AufBen13}  or variance \cite{EstLeo14,Mar15} of the \EC, denoted by $\chi $ here.

As an integer-valued  quantity, the \EC~can be easily measured and used in many estimation and modelisation procedures. It is an important indicator of the porosity of a random media \cite{AMMS,Hil02, SWRS}, it is used in brain imagery \cite{KilFri,TayWor}, astronomy, \cite{Mar15,Mel90,Schmalzing}, and many other disciplines. 
 See also \cite{ABBSW} for a general review of applied algebraic topology.

 Most of the available works on random fields use the results  gathered in the celebrated monograph \cite{AdlTay07}, or similar variants. In this case, theoretical computations of the \EC~emanate from Morse theory, where the focus is   on the local extrema of the underlying field   instead of the set itself. For the theory to be applicable, the functions must be $\C^{2}$ and satisfy the Morse hypotheses, which conveys some restrictions on the set itself. 

 The expected Euler characteristic  also turned out to be a widely used approximation of the distribution function of the maximum of a Morse random field, and attracted much interest in this direction, see \cite{AdlSam, AufBen13,AzaWsc,TayWor}.
 Indeed, for large $r>0$, a well-behaved field rarely exceeds $r$, and if it does, it is likely to have a single highest peak, which yields that the level set of $f$ at level $r$, when not empty,  is most often simply connected, and has Euler characteristic $1$. Thereby, $\E \chi (\{f\geqslant   r\})\approx\P(\sup f\geqslant   r)$, which provides an additional motivation to compute the mean \EC~of random fields.

Even though \cite{AST} provides an asymptotic expression for some classes of  infinitely divisible fields, most of the tractable formulae concern Gaussian fields. One of the ambitions of this paper is to provide a formula that is tractable in a rather general setting, and also works in the Gaussian realm. There seems to be no particular obstacle to extend these ideas to higher dimensions in a future work.


Given a set $A\subset \mathbb{R}^{2}$, let $\Gamma (A)$ be the class of its bounded  arc-wise connected components. We say that a set $A$ is \textit{admissible}  if $\Gamma (A)$ and $\Gamma (A^{c})$ are finite, and in this case its \EC~is defined by 
\begin{align*}
\chi (A)=\#\Gamma (A)-\#\Gamma (A^{c}),
\end{align*}
where $\#$ denotes the cardinality of a set.  The theoretical results of Adler and Taylor \cite{AdlTay07} regarding the \EC~of random excursions require second order differentiability of the underlying   field $f$, but the expression of the mean \EC~only involves the first-order derivatives, suggesting that second order derivatives do not matter in the computation of the \EC.
In the words of Adler and Taylor (Section 11.7), regarding their Formula (11.7.6), it is a
{\it rather surprising fact that the [mean \EC~of a Gaussian field] depends on the covariance of $f$ only through some of its derivatives at zero}, the latter referring to first-order partial derivatives. We present here a new method for which the second order differentiability is not needed. The results are valid for  $\mathcal{C}^{1}$ fields with locally Lipschitz derivatives, also called $\mathcal{C} ^{1,1}$ fields, relaxing slightly the classical $\mathcal{C} ^{2}$ Morse framework.

Our   results  exploit the findings of \cite{LacEC1} connecting smooth sets \EC~and variographic tools. For some $\lambda \in \mathbb{R}$ and a   bi-variate function $f$,  define {for $x\in \mathbb{R}^{2}$} the event
\begin{align*}
\delta  ^{\eta  }(x,f,\lambda )=\mathbf{1}_{\{f(x)\geqslant \lambda ,f(x+\eta  \u _{1})<\lambda ,f(x+\eta \u_{2} )<\lambda \}}, \eta  \in \mathbb{R},
\end{align*}
where $(\u_{1},\u_{2})$ denotes the canonical basis of $\mathbb{R}^{2}$, assuming $f$ is defined in these points.  {When $f$ is a random field, let $\bar \delta  ^{\eta  }(x,f,\lambda )$  denote the event $\{ \delta ^{\eta }(x,f,\lambda )=1\}$}. Let us write a corollary of our main result here, a more general statement can be found in Section  \ref{sec:random-excursions}. {Denote by $\Vol^2 $ the Lebesgue measure on $\mathbb{R}^{2}$}. For $W\subset \mathbb{R}^{2}$ and a function $f:W\to \mathbb{R}^{2}$, introduce the mapping $\mathbb{R}^{2}\to \mathbb{R}^{2}$,
\begin{align*}
f_{[W]}(x)=\begin{cases}-\infty$ if $x\notin W\\
f(x)$ otherwise$, 
\end{cases}
\end{align*} 
so that the intersections of level sets of $f$ with $W$ are the level sets of $f_{[W]} .$
\begin{theorem}
Let $W=[0,a]\times [0,b]$ for some $a,b>0$,  $f$ be a $\mathcal{C}^{1}$ real random field  on $\mathbb{R}^{2}$ with locally Lipschitz partial derivatives $\partial _{1}f,\partial _{2}f$, $\lambda \in \mathbb{R}$, and {let} $F=\{x\in W:f(x)\geqslant \lambda \}$. Assume furthermore that the following conditions are satisfied: \begin{itemize}
\item [(i)]   For some   $\kappa  >0$, for $x\in \mathbb{R}^{2}$,   the random  vector $(f(x),\partial _{1}f (x),\partial_{2}f (x))$ has a density  bounded   by   $\kappa  $ from above on $\mathbb{R}^{3}$.
\item [(ii)]  There is $p>6 $ such that 
\begin{align*}
  \E[ \Lip(f,W)^{p}]<\infty ,\;\E[\Lip( \partial_i f,W)^{p }]<\infty, i=1,2,
\end{align*} 
where $\Lip({g},W )$ denotes the Lipschitz constant {of a vector-valued function $g$} on $W$.
\end{itemize}
Then   $\E[\#\Gamma (F )]<\infty ,\E[\#\Gamma (F^{c})]<\infty ,$   and
\begin{align}
\label{eq:expect-EC-excursion}
\E[ \chi (F )]&= \lim_{\varepsilon \to 0} \sum_{x\in \varepsilon \mathbb{Z} ^{2}}[\P(\bar \delta^{\varepsilon }(x,f_{[W]} ,\lambda ) )-\P(\bar \delta^{-\varepsilon }(x,-f_{[W]} ,-\lambda ) )]\\
\label{eq:expect-EC-covariogram}&=\lim_{\varepsilon \to 0}\varepsilon ^{-2}\int_{\mathbb{R}^{2}}\Big[\P(\bar \delta^{\varepsilon }(x,f_{[W]} ,\lambda ) )-\P(\bar \delta^{-\varepsilon }(x,-f_{[W]} ,-\lambda ) )\Big]dx.
\end{align}{
If $f$ is furthermore stationary, we have
\begin{align*}
\E [\chi (F)]=\overline{\chi }(f,\lambda )\Vol  ^2  (W)+\overline{\Per}(f,\lambda )\Per(W)+\overline{\Vol ^2 }(f,\lambda )\chi (W)
\end{align*}where the volumic Euler characteristic, perimeter and volume $\overline{\chi },\overline{Per},\overline{\Vol^{2}}$ are defined in Theorem \ref{th:stationary-fields}, they only depend on the behavior of $f$ around the origin.}
\end{theorem}

 The right hand side of \eqref{eq:expect-EC-covariogram} is related to the \emph{bicovariogram} of the set $F$, defined by 
\begin{align}
\label{eq:bicovariogram}
\delta _{0}^{x,y}(F)=\Vol  ^2  (F\cap (F+x)^{c}\cap (F+y)^{c}), x,y\in \mathbb{R}^{2},
\end{align}
in that \eqref{eq:expect-EC-covariogram} can be reformulated as 
\begin{align*}
\E\chi (F )=\lim_{\varepsilon \to 0}\varepsilon ^{-2}(\E[\delta _{0}^{-\varepsilon \u_{1},-\varepsilon \u_{2}}(F )]-\E[\delta _{0}^{\varepsilon \u_{1},\varepsilon \u_{2}}(F^{c} )]).
\end{align*} 
This approach seems {to be }new in the literature. It highlights the fact that under suitable conditions, the mean \EC~of random level sets is {linear in} the field's third order marginal. In \cite[Corollary 6.7]{Fu}, Fu gives an expression for the Euler characteristic of a set with positive reach  by means of local topological quantities related to the height function. If the set is the excursion of a random field, this approach is of a different nature, as passing Fu's formula to  expectations would not lead to  an expression depending directly on the field's marginals.

We also give in Theorem \ref{th:bound-pixel-comp-level-set} a bound on the number of connected components of the excursion of $f$, valid in any dimension, which is finer than just bounding by the number of critical points; we could not locate an equivalent result in the literature. This topological estimate is interesting in its own and also applies uniformly to the number of components of 2D-pixel approximations of the excursions of $f$. We therefore  use it here as a  {majoring bound} in the application of Lebesgue's theorem to obtain \eqref{eq:expect-EC-excursion}-\eqref{eq:expect-EC-covariogram}.

It is likely that the results concerning the planar Euler characteristic could be extended to higher dimensions.  See for instance \cite{Sva14}, that paves the way to an extension of the results of \cite{LacEC1} to  random fields on spaces with arbitrary dimension. Also, the uniform bounded density hypothesis is relaxed and allows for the density of the $(d+1)$-tuple $(f(x)  ,\partial _{1}f(x),\dots ,\partial _{d}f(x))$ to be arbitrarily large in the neighborhood of $(\lambda ,0,\dots ,0)$.
Theorem \ref{th:conditions-random-excursion}  features a result where $f$ is defined on the whole space and the level sets are observed through a bounded window $W$, as is typically the case for  level sets of  non-trivial stationary fields, but the intersection with $\partial W$ requires additional notation and care.  See Theorem \ref{th:stationary-fields} for a result tailored to deal with excursions of stationary fields.  

Theorem \ref{th:gauss-general} features the case where $f$ is a Gaussian field  assuming only $\mathcal{C}^{1,1}$ regularity (classical literature {about random excursions} require $\mathcal{C}^{2}$ Morse fields in dimension $d\geqslant 2$, or $\C^{1}$ fields in dimension $1$). Under the additional hypothesis that $f$ is stationary and isotropic, we retrieve in Theorem \ref{th:Gauss-stationary} the classical results of \cite{AdlTay07}.\\
  
  Let us explore other consequences of our results. Let $h:\mathbb{R}\to \mathbb{R}$ be a $\mathcal{C}^{1}$ test function with compact support, {and $F$ as in Theorem \ref{eq:expect-EC-excursion}}. Using the  results  of our paper, it is shown in the follow-up article  \cite{Lac-KacRice}  that for any deterministic $\mathcal{C}^{2}$ Morse function $f$ on $\mathbb{R}^{2}$, 
   \begin{align}
   \label{eq:kac-rice}
 \int_{\mathbb{R}}\chi (F)h(\lambda )d\lambda =-\sum_{i=1}^{2}\int_{W}\mathbf{1}_{\{\nabla f(x)\in Q_{i}\}}[h'(f(x))\partial _{i}f(x)^{2}+h(f(x))\partial _{ii}f(x)]dx+\text{\rm{boundary terms}}
\end{align}
where 
\begin{align*}
Q_{1}=\{(x,y)\in \mathbb{R}^{2}:y<x<0\},\; Q_{2}=\{(x,y)\in \mathbb{R}^{2}:x<y<0\},
\end{align*}
yielding applications for instance to shot-noise processes. In the context of random functions, {no marginal density hypothesis  is required to take the expectation}, at the contrary of analogous results, including those from the current paper.
 Bierm\'e \& Desolneux \cite[Section 4.1]{BieDesEuler} later gave another interpretation of \eqref{eq:kac-rice}, showing   that if it is extended to a random isotropic stationary field which gradient does not vanish a.e. a.s., it can be rewritten as a  simpler expression, after appropriate integration by parts, namely  
\begin{align*}
\E\left[
 \int_{U}\chi (\{f\geqslant \lambda \};U)h(\lambda )d\lambda
\right] = \Vol  ^2  (U)\E\left[
h(f(0)) [ -\partial _{11}f(0) +4 \partial _{12}f(0)\partial _{1}f(0)\partial _{2}f(0)\|\nabla f(0)\|^{-2} ]
\right],
\end{align*}
where $U$ is an appropriate open set, and $\chi (\{f\geqslant \lambda ;U\})$ is the total  curvature of the level set $\{f\geqslant \lambda \}$ within $U$, generalizing the Euler characteristic. They obtained this result by totally different means, via an approach involving Gauss-Bonnet theorem,   without any requirement on $f$ apart from being $\mathcal{C}^{2}$.

     \section{Topological approximation}
     \label{sec:regular-functions}

Let $f$ be a function of class $\mathcal{C}^{1}$ over some window $W\subset \mathbb{R}^{d}$, and $\lambda \in \mathbb{R}$. Define
\begin{align*}
F:=F_{\lambda}(f)=\{x\in W:f(x)\geqslant  \lambda \},\hspace{1cm}F_{\lambda ^{+}}(f)=\{x\in W:f(x)>\lambda \}.
\end{align*} Remark that $F_{\lambda ^{+}}(f)=(F_{-\lambda }(-f))^{c}$. If we assume that $\nabla f$ does not vanish on $f^{-1}(\{\lambda \})$, then $\partial F_{\lambda }(f)=\partial F_{\lambda ^{+}}(f)=f^{-1}(\{\lambda \})$, and this set is furthermore Lebesgue-negligible, as a $(d-1)$-dimensional  manifold.

 According to  \cite[4.20]{Fed}, $\partial F_{\lambda }(f)$ is regular in the sense that { its boundary is $\C^{1}$ with Lipschitz normal, }if  $\nabla f$ is locally Lipschitz and does not vanish on $\partial F_{\lambda }(f)$. This condition is necessary to prevent $F$ from having locally infinitely many connected components, which would make \EC~not properly defined in dimension $2$,    see \cite[Remark 2.11]{LacEC1}. We call $\mathcal{C}^{1,1}$ function a differentiable function {whose} gradient is a locally Lipschitz mapping. Those functions have been mainly used in optimization problems, and as solutions of some PDEs, see for instance \cite{HSN}. They can also be characterized as the functions which are locally \textit{semiconvex} and \textit{semiconcave}, see \cite{CanSin}. 
 
The results of \cite{LacEC1} also yield that the Lipschitzness of $\nabla f$ is  sufficient for the digital  approximation of $\chi (\{f\geqslant \lambda \}) $ to be valid. It seems therefore that the $\C^{1,1}$ assumption  is the minimal one ensuring   the \EC~to be computable in this fashion.

\subsection{Observation window} An aim of the present paper is to advocate the power of variographic tools for computing intrinsic volumes of random fields excursions. Since many applications are  concerned with stationary random fields on the whole plane, we have to study the intersection of excursions with bounded windows, and assess the quality of the approximation.

To this end, call \emph{rectangle} of $\mathbb{R}^{d}$  any set $W=I_{1}\times \dots \times I_{d}$ where the $I_{k}$  are possibly infinite closed intervals of $\mathbb{R}$ with non-empty interiors, and let $\corners(W )$, which number is between $0$ and $2^{d}$, be the points having extremities of the $I_{i}$ as coordinates. Then call \textit{polyrectangle} a finite union 
 $
W=\cup _{i}W_{i}
$ where each $W_{i}$ is a rectangle, and for $i\neq j, \corners(W _{i})\cap \corners(W _{j})=\emptyset $. Call $\W_{d}$  the class of   polyrectangles. 

  For  $W\in \W_{d}$ and $x\in W$, let $I_{x}(W)=\{1,\dots ,d\}$ if $x\in \text{\rm{int}}(W)$, and otherwise let  $I_{x}(W)\subset \{1,\dots ,d\}$ be the set of indices $i$ such that  $x+\varepsilon \u_{i}\in \partial W$ and $x-\varepsilon \u_{i}\in \partial W$ for arbitrarily small $\varepsilon >0$, where $\u_{i}$ is the $i$-th canonical vector of $\mathbb{R}^{d}$. Say then that $x$ is a $k$-dimensional point of $W$ if $ | I_{x}(W) | =k.$ Denote by $\partial _{k}W$ the set of $k$-dimensional points, and call $k-$dimensional facets the connected  components of $\partial _{k}W$. Remark that $I_{x}(W)$ is constant over a given facet.  Note that $\partial _{d}W=\text{\rm{int}}(W)$ and $\partial W=\cup _{k=0}^{d-1}\partial _{k}W$. We extend the notation $\text{\rm{corners}}(W)=\partial _{0}W.$ An alternative definition is that a subset $F\subset W$ is a   facet of $W$ if it is a maximal relatively open subset of a   affine subspace of $\mathbb{R}^{d}$.


\begin{definition}
Let $W\in \W_{d}$, and $f:W\to \mathbb{R}$ be of class ${\C^{1,1}}$. Say that   $f$ is \emph{regular within $W$  } at some level $\lambda \in \mathbb{R}$  if for $0\leqslant k\leqslant d$, $\{x\in \partial _{k}W:f(x)=\lambda ,\partial _{i}f(x)=0,i\in I_{x}(W)\}=\emptyset $, or equivalently if for every $k$-dimensional  facet $G$ of $W$, the $k$-dimensional gradient of the restriction of $f$ to $G$ does not vanish on $f^{-1}(\{\lambda \})\cap G$.
\end{definition}

For such a function $f$ in dimension $2$, it is shown in \cite{LacEC1} that  the Euler characteristic of its excursion set $F=F_{\lambda }(f)\cap W$    can be expressed by means of its bicovariograms, defined in \eqref{eq:bicovariogram}.  For $\varepsilon>0 $ sufficiently small 
\begin{align} \label{eq:deterministic-expression-chi}
\chi (F ) &=  \varepsilon ^{-2}[\delta _{0}^{-\varepsilon \u_{1},-\varepsilon \u_{2}}(F )-\delta _{0}^{\varepsilon \u_{1},\varepsilon \u_{2}}( F  ^{c})].
\end{align} 
The proof is based on the Gauss approximation  of $F$:
\begin{align*}
F^{\varepsilon }=\bigcup _{x\in \varepsilon \mathbb{Z} ^{2}\cap F}\left(
x+\varepsilon [-1/2,1/2)^{2}
\right).
\end{align*}
  According to \cite[Theorem 2.7]{LacEC1}, 
  for $\varepsilon $ sufficiently small, 
\begin{align*}
\chi (F )&=\chi ( F ^{\varepsilon })\\
&=\sum_{x\in \varepsilon \mathbb{Z} ^{2}}(\delta  ^{\varepsilon }(x,f_{[W]},\lambda )-\delta  ^{-\varepsilon }(x,-f_{[W]},-\lambda ))\\
&=\varepsilon ^{-2}\int_{\mathbb{R}^{2}}(\delta  ^{\varepsilon }(x,f _{[W]},\lambda )-\delta ^{-\varepsilon }(x,-f _{[W]},-\lambda ))dx.
\end{align*}  If $f$ is a random field, the difficulty to pass the result to expectations is to majorize the right hand side uniformly in 
 $\varepsilon $ by an integrable quantity, and this goes through bounding the number of connected components of $F $ and its approximation $F ^{\varepsilon }$. This is the object of the next section.

 \subsection{Topological estimates}

  The next result, valid in  dimension $d\geqslant 1,$   does not concern directly the Euler characteristic. Its purpose is to bound the number of connected components of $F_{\lambda }(f)\cap W$ by an expression depending on $f$ and its partial derivatives. It turns out that a similar bound holds for the excursion approximation $(F_{\lambda}(f)\cap W)^{\varepsilon }$ in dimension $2$, uniformly in $\varepsilon $, enabling the application of Lebesgue's theorem to the point-wise convergence \eqref{eq:deterministic-expression-chi}.  

Traditionally, see for instance \cite[Prop. 1.3]{EstLeo14}, the number of connected components of the excursion set, or its Euler characteristic, is  bounded  by using the number of critical points, or by the number of points on the level set where $f$'s gradient points towards a predetermined direction.  Here,  we use another method based on the idea that in a small connected component, a critical point is necessarily close to the boundary, where $f-\lambda $ vanishes. It yields the expression \eqref{eq:bound-gamma-level-set-continuous} as a bound on the number of connected components. It also allows  in Section \ref{sec:random-excursions}, devoted to random fields, to relax the usual uniform density assumption on the marginals of the $(d+1)$-tuple $(f ,\partial _{i}f,i=1,\dots ,d)$, leaving the possibility that the density is unbounded around $(\lambda ,0,\dots ,0)$.

Denote by $\Lip(g;A)\in \mathbb{R}_{+}\cup \{\infty \}$, or just $\Lip(g),$ the Lipschitz constant of a mapping $g$ going from a metric space $A$ to another metric space.   Let $W\in \W_{d}, g:W\to \mathbb{R}$,   $\C^{1}$ with Lipschitz derivatives.  Denote by $\mathcal{H}_{d}^{k}$ the $k$-dimensional Hausdorff measure in $\mathbb{R}^{d}$.
 Define the possibly infinite quantity, for $1\leqslant k\leqslant d,$
\begin{align*}
I_{k}(g;W) :&=  \max(\Lip(g),\Lip(\partial _{i}g)   ,1\leqslant i\leqslant d)^{k}
\int_{\partial _{k}W}  \frac{\mathcal{H}_{d}^{k}(dx)}{
\max\left(
  | g(x) | ,| \partial _{i}g(x) | 
,i\in I_{x}(W)\right)^{ k}}
,
\end{align*}  and $I_{0}(g;W):=\#\text{\rm{corners}}(W)$. {Put $I_{k}(g;W)=0$ if $\Lip(g)=0$ and $g$ vanishes}, $1\leqslant k\leqslant d$.
  \begin{theorem}
\label{th:bound-pixel-comp-level-set}

Let $W\in \W_{d}$, and $f:W\to \mathbb{R}$ be a $\mathcal{C} ^{1,1}$ function. Let $F=F_{\lambda }(f)$ or $F=F_{\lambda ^{+}}(f)$ for some $\lambda \in \mathbb{R}$. Assume that $f$ is regular at level $\lambda $ in $W$. \begin{itemize}
\item  
[{(i)}] For $d\geqslant 1,$ \begin{align}
\label{eq:bound-gamma-level-set-continuous}
\#\Gamma (F\cap W)&\leqslant   \sum_{k=0}^{d}2^{k} \kappa_{k} ^{-1}I_{k}(f-\lambda;W ),
 \end{align}where $\kappa _{k}$ is the volume of the $k$-dimensional unit ball. 

\item [{(ii)}] If $d=2$,  
\begin{align}
\label{eq:bound-gamma-level-set-pixel}
\#\Gamma ((F\cap W)^{\varepsilon }) \leqslant C\sum_{k=0}^{2}I_{k}(f-\lambda ;W)
\end{align}for some $C>0$ not depending on $f,\lambda,$ or $\varepsilon  $.
\end{itemize}
\end{theorem}
    The proof is given in Section \ref{sec:proofs}.
  
   \begin{remark} Theorem \ref{th:conditions-random-excursion} gives conditions on the marginal densities of a bivariate random field ensuring  that the term on the right hand side has finite expectation.
\end{remark}

\begin{remark} Similar results hold if   partial derivatives of $f$ are only assumed to be H\"older-continuous, i.e. if there is $\delta >0$ and $H_{i}>0,i=1,\dots ,d,$ such that $\|\partial _{i}f(x)-\partial _{i}f(y)\|\leqslant H_{i}\|x-y\|^{\delta }$ for $x,y$ such that $[x,y]\subset W$. Namely, we have to change constants and replace the exponent $k$ in the  {\it max} by an exponent  $k\delta $. We do not treat such cases here because, as noted at the beginning of Section \ref{sec:regular-functions}, if the partial derivatives are not Lipschitz, the upper level set is not regular enough to compute the \EC~from the bicovariogram, but the proof is similar to the $\mathcal{C}^{1,1}$ case.
\end{remark}

\begin{remark}
\label{rk:bound-levelset}
Calling $B $ the right hand term of \eqref{eq:bound-gamma-level-set-pixel} and noticing that $F_{\lambda^{+} }(f)^{c}$ is an upper level set of $-f$, an easy reasoning yields (see \cite[Remark 2.13]{LacEC1})
\begin{align*}
  | \chi ((F_{\lambda }(f)\cap W)^{\varepsilon }) |&\leqslant  2B. 
\end{align*}
\end{remark}


\section{Mean \EC~of random excursions}
\label{sec:random-excursions}

We call here $\C^{1}$ random field over a set $\Omega \subseteq \mathbb{R}^{d}$  a {separable random field} $ (f(x);x\in \Omega ) $, such that in each point $x\in \Omega $, the limits 
\begin{align*}
\partial _{i}f(x):=\lim_{s \to 0}\frac{f(x+s\u_{i})-f(x)}{s},\;i=1,2,
\end{align*}exist a.s., and the fields $(\partial _{i}f(x),x\in \Omega),i=1,\dots ,d$, are a.s. {separable with continuous sample paths}. See \cite{Adl81,AdlTay07} for a discussion on the regularity properties of random fields.
Say that the random field is $\mathcal{C}^{1,1}$ if the partial derivatives are a.s. locally Lipschitz.

Many sets of  conditions allowing to take the expectation in \eqref{eq:deterministic-expression-chi} can be derived from Theorem  \ref{th:bound-pixel-comp-level-set}. We give below a compromise between optimality and compactness.

\begin{theorem}
\label{th:conditions-random-excursion}Let $W\in \W_{d}$ {bounded},
and let $f$ be a $\mathcal{C}^{1,1}$ random field on $W$, $\lambda \in \mathbb{R}, {F=\{x\in W:f(x)\geqslant \lambda \}}$.  Assume that the following conditions are satisfied: \begin{itemize}
\item [(i)]   For some   $\kappa  >0 ,\alpha >1$, for $1\leqslant k\leqslant d,x\in\partial _{ k} W,I\subset \mathcal{I}_{ k}$,  the   random  $( k+1)$-tuple $(f(x)-\lambda ,\partial _{i}f (x),i\in I)$ satisfies
\begin{align*}
\P( | f(x) -\lambda | \leqslant \varepsilon , | \partial _{i}f(x) | \leqslant \varepsilon ,i\in I)\leqslant \kappa \varepsilon ^{\alpha {k}},\varepsilon >0,
\end{align*} \item [(ii)]  for some $p>  d\alpha  (\alpha   -1)^{-1},  $
\begin{align*}
  \E[ \Lip(f)^{  p}]<\infty ,\;\E[\Lip(\partial_i f)^{p  }]<\infty,i=1,\dots ,d .
\end{align*} 
\end{itemize}
Then    $ \E[\#\Gamma (F )]<\infty ,\E[\#\Gamma (F^{c})]<\infty $ and $f$ is a.s. regular within $W$ at level $\lambda $.
 In the context $d=2$,  
\eqref{eq:expect-EC-excursion}-\eqref{eq:expect-EC-covariogram}  give the mean Euler characteristic.

\end{theorem}

\begin{remark} 
In the case where the $\Lip(f),\Lip(\partial _{i}f),i=1,\dots ,d$ have a finite  moment of  order $>d(d+1)$, the hypotheses are satisfied if for instance $(f(x)-\lambda ,\partial _{i}f(x),1\leqslant i\leqslant d)$ has a uniformly bounded multivariate density, in which case $\alpha =(d+1)/d$ is suitable. If $\alpha <(d+1)/d$,    higher  moments for the Lipschitz constants are required.
\end{remark}

 The proof is deferred to Section \ref{sec:proofs}. We give an explicit expression in   the case where $f$ is stationary. Boundary terms  involve the perimeter  of $F$, introduced below. Denote by $\C_{c}^{1}$ the class of compactly supported $\C^{1}$ functions on $\mathbb{R}^{2}$ {endowed with the norm $\|\varphi \|=\sup_{x\in \mathbb{R}^{2}} | \varphi (x) | $}. For a measurable set $A\subset \mathbb{R}^{2}$, and $\u\in \S^{1}$, the unit circle in $\mathbb{R}^{2}$, define the variational perimeter of $A$ in direction $\u$ by
\begin{align*}
\Per_{\u}(A)=\sup_{\varphi \in \mathcal{C}_{c}^{1 }:\|\varphi \|\leqslant 1  }\int_{A}\langle \nabla \varphi(x) ,\u\rangle dx.
\end{align*}{Recall that $(\u_{1},\u_{2})$ is the canonical basis of $\mathbb{R}^{2}$, and introduce }the $\|\cdot \|_{\infty }$-perimeter 
\begin{align*}
\Per_{\infty }(A)=\Per_{\u_{1}}(A)+\Per_{\u_{2}}(A),
\end{align*}named so because it is the analogue of the classical perimeter when the Euclidean norm is replaced by the $\|\cdot \|_{\infty }$-norm, see \cite{GalLac}.
 \begin{theorem}
 \label{th:stationary-fields}
Let $f$ be a $\mathcal{C}^{1,1}$ stationary random field on $\mathbb{R}^{2},$ $\lambda \in \mathbb{R}$,  $W\in \W_{2}$ bounded. Assume that $(f(0),\partial _{1} f(0),\partial _{2}f(0))$ has a bounded density, and that there is $p>6 $ such that  \begin{align*}
\E\Big[\Lip(f;W)^{ p}\Big] <\infty ,\;\E\Big[\Lip(\partial _{i}f;W)^{p }\Big]<\infty ,i=1,2.
\end{align*}
Then the following limits exist:  
\begin{align*}
\overline{\chi }(f{,\lambda }  )&:=\lim_{\varepsilon \to 0}\varepsilon ^{-2}\Big[\P(\bar \delta^{\varepsilon }(0,f,\lambda ))-\P(\bar \delta^{-\varepsilon }(0,-f,-\lambda))\Big],\\
\overline{\Per_{\u_{i}}}(f {,\lambda } )&:=\lim_{\varepsilon \to 0}\varepsilon ^{-1}\P(f(0)\geqslant \lambda ,f(\varepsilon \u_{i})<\lambda ),i=1,2,\\
\overline{\Vol ^2 }(f{,\lambda }  )&:=\P(f(0)\geqslant \lambda ),
\end{align*}and we have, with $\overline{\Per_{\infty }}=\overline{\Per_{\u_{1}}}+\overline{\Per_{\u_{2}}},$  
\begin{align} 
\notag\E [\chi (F_{\lambda }(f)\cap W)]&=\Vol ^2 (W)\overline{\chi} (f{,\lambda })+\frac{1}{4}(\Per_{\u_{2}}(W)\overline{\Per_{\u_{1} }}(f{,\lambda })+\Per_{\u_{1}}(W)\overline{\Per_{\u_{2} }}(f{,\lambda }))\\
\label{eq:stationary-expectation-EC}&\hspace{4cm}+\chi (W)\overline{\Vol ^2 }(f{,\lambda })\\
\label{eq:stationary-expectation-Per}\E[ \Per_{\infty } (F_{\lambda }(f)\cap W)]&=\Vol ^2 (W)\overline{\Per_{\infty }} (f{,\lambda })+\Per_{\infty }(W)\overline{\Vol ^2 }(f {,\lambda }) \\
\label{eq:stationary-expectation-Vol}
\E [\Vol ^2 (F_{\lambda }(f)\cap W)]&=\Vol ^2 (W)\overline{\Vol ^2 }(f{,\lambda }).
\end{align}
  \end{theorem}

The proof of this result requires notation contained in the proof of Theorem \ref{th:conditions-random-excursion}, it is therefore placed at Section \ref{sec:proof-4-2}.
 

\renewcommand{\v}{\mathbf{v}}

  \section{Gaussian level sets}
\label{sec:gauss} 

 Let $(f(x),x\in W)$ be a centred  Gaussian field on some $W\in \W_{d}$.  Let the covariance function be defined by
\begin{align*}
\sigma (x,y)=\E[ f(x)f(y)],\;x,y\in W.
\end{align*}
Say that some  real function $h$ satisfies the Dudley condition on $D\subset W$ if for some $\alpha  >0$, $ | h(x)-h(y) | \leqslant  | \log(\|x-y\|) | ^{-1-\alpha }$ for $x,y\in W$.
 We will make the following assumption on $\sigma $:
\begin{assumption}
\label{ass:ass-cov}
 {Assume that     $x\in W\mapsto \partial ^{2}\sigma (x,x)  /\partial x_{i}\partial y_{i}$ exists and satisfies the Dudley condition for $ i=1,\dots ,d,$   that  the partial derivatives
 $\partial  ^{4}\sigma(x,x) /\partial _{x_{i}}\partial _{x_{j}}\partial _{y_{i}}\partial _{y_{j}},x\in W$, $1\leqslant i,j\leqslant d$, exist   and that for some finite partition $\{D_{k}\}$ of $W$ they satisfy the Dudley condition over each $D_{k}$.  }
\end{assumption}  
 \begin{theorem}
\label{th:gauss-general}
Let $W\in \W_{d}$ bounded. {Assume that $\sigma $ satisfies Assumption \ref{ass:ass-cov}} and  that for $x\in W$, $(f(x),\partial_i f(x),i=1,\dots ,d)$ is non-degenerate.
 Then for any $\lambda \in \mathbb{R}$, $F=F_{\lambda }(f)$ satisfies the conclusions of Theorem \ref{th:conditions-random-excursion}.
\end{theorem}

\begin{proof} { 
Assumption \ref{ass:ass-cov} and  \cite[Theorem 2.2.2]{Adl81} yield that for $i=1,\dots ,d$, $ (\partial _{i} f(x);x\in W) $ is well defined in the $L^{2}$ sense and is a Gaussian field with covariance functions $  \E[ \partial _{i}f(x)\partial _{i}f(y)]=\partial ^{2}\sigma (x,y) /\partial x_{i}\partial y_{i}$  for $x,y\in W $.  
Since the latter covariance functions satisfy Dudley condition, Theorem 1.4.1 in \cite{AdlTay07} implies   the sample-paths continuity of the partial derivatives.

 Using again  \cite[Theorem 2.2.2]{Adl81}, for $1\leqslant i,j\leqslant d$, $( \partial _{i,j}f(x),x\in D) $ is a well-defined Gaussian field with covariance $\E[ \partial _{i,j}f(x)\partial _{i,j}f(y)]=\partial ^{4}\sigma (x,y)/\partial x_{i}\partial y_{i}\partial x_{j}\partial y_{j}$. 
 For each $k$,  \cite[Theorem 1.4.1]{AdlTay07} again yields that $\partial _{i,j}f$ is continuous and bounded over $D_{k}$, hence $\partial _{i,j}f$ is bounded over $W$. Finally, formula (2.1.4) in \cite{AdlTay07} yields that  $\E\left[
 \sup_{x\in W} | \partial _{i,j}f (x) | ^{p}
\right]<\infty  $ for $p\geqslant 0$. Since $\Lip(\partial _{i}f)\leqslant d\max_{j=1,\dots ,d}\|\partial _{ij}f\|$, Condition (ii) of Theorem \ref{th:conditions-random-excursion} is    satisfied for any $\alpha >1$.
}

To prove (i), put for notational convenience  $f^{(0)}:=f,f^{(i)}=\partial _{i}f, i=1,\dots ,d$.
 We have for $i,j\in \{0,\dots ,d\}$,
\begin{align*}
 | \E    f^{(i)}(x)f^{(j)}(x)&-f^{(i)}(y)f^{(j)}(y) |\\
 &\leqslant \left|
  \E \left[\left(  f^{(i)}(x)-f^{(i)}(y)\right)f^{(j)}(x)\right] 
\right|
+\left|
  \E  \left[
 f^{(i)}(y)\left(
f^{(j)}(x)-f^{(j)}(y)
\right)
\right]
\right|
\\
 &\leqslant \E\left[
 \sup_{W} | f^{(j)} |\Lip( f^{(i)})
\right]\|x-y\|+\E\left[
 \sup_{W} |  f^{(i)} |\Lip(f^{(j)})
\right]\|x-y\|,   
\end{align*}which yields that the covariance function with values in the space of $(d+1)\times (d+1)$ matrices, 
\begin{align*}
x\mapsto \Sigma (x):=\text{cov}(f(x),\partial_i f(x),1\leqslant i\leqslant d)
\end{align*} is Lipschitz on $W$. In particular, since $\det(\Sigma (x))$ does not vanish on $W$, it is bounded from below by some $c >0$, whence the density of $(f(x),\partial_1 f(x),\partial_2 f(x))$, $x\in W$, is uniformly bounded by $(2\pi )^{-d/2}c^{-1/2}$, and assumption (i) from Theorem  \ref{th:conditions-random-excursion} is   satisfied with $\alpha =(d+1)/d$.  \end{proof}

{ 

\begin{example}Random fields that are $\C^{1,1}$ and not $\mathcal{C}^{2}$ naturally arise in the context of smooth interpolation. 
Let $E=\{x_{i};i\in \mathbb{Z} \} $ be a countable set of   points of $\mathbb{R}$, such that $x_{i}<x_{i+1},i\in \mathbb{Z} $. Let $(W(x ),x\in E)$ be a random field on $E$, and $A_{x},B_{x},x\in E $ be random variables on the same probability space.  Define
\begin{align*}
g (y)=\sum_{i\in \mathbb{Z}  }\mathbf{1}_{\{y\in [x_{i},x_{i+1})\}}\left[
A_{x_{i}}\left(
\frac{ y-x_{i} }{x_{i+1}-x_{i}}
\right)^{2}+B_{x_{i}} \frac{y-x_{i}}{x_{i+1}-x_{i}}+W(x_{i}) 
\right].
\end{align*}  Straightforward computations yield that, with $\Delta _{x}=W(x_{i+2})-2W(x_{i+1})+W(x_{i})$, if
\begin{itemize} 
\item $A_{x_{i+1}}=\Delta _{x _{i}}-A_{x_{i}},i\in \mathbb{Z}  ,$
\item $B_{x_{i}}=W(x_{i+1})-W(x_{i})-A_{x},i\in \mathbb{Z}  ,$
\end{itemize}then   with probability $1$, $g$ is a $\C^{1,1}$ and in general not twice differentiable field on  $(\lim_{i\to -\infty }x_{i},\lim_{i\to \infty  }x_{i})$ such that $g(x_{i})=W(x_{i}),i\in \mathbb{Z} $. If for some $i_{0}\in \mathbb{Z} ,$ $ (A_{x_{i_{0}}};W(x_{i}),i\in \mathbb{Z} )$ is a Gaussian process, $g$ is furthermore a Gaussian field. 

  Given a Gaussian process  $(g(k);k\in \mathbb{Z} ^{d})$,  it should be possible to carry out a similar approximation scheme in $\mathbb{R}^{d}$  by defining  $g =\sum_{k\in \mathbb{Z}^{d}}\mathbf{1}_{\{x\in (k+[0,1)^{d})\}}g_{k}$ where $g_{k} $ is a bicubic  polynomial interpolation of Gaussian variables $W(j),j\in (k+\{0,1\}^{d})$ on $k+[0,1)^{d}$. A possible follow-up of this work could be to investigate the asymptotic properties of topological characteristics of $g$ when it is the smooth interpolation of an irregular Gaussian field as the grid mesh converges to $0.$

\end{example}
}
 
Let us give the mean \EC~in dimension $2$ under the simplifying assumptions that the law of $f$ 
is invariant under translations and rotations of $\mathbb{R}^{2}$. 
 This implies for instance that in every $x\in \mathbb{R}^{2}$, $f(x),\partial _{1}f(x)$ and $\partial _{2}f(x)$ are independent, see for instance \cite{AdlTay07} Section 5.6 and (5.7.3).  Assumption \ref{ass:ass-cov} is simpler to state in this context: $x\mapsto \partial ^{2}\sigma (x,x)/\partial x_{i}\partial y_{i}$ and $x\mapsto \partial ^{4}\sigma (x,x)/\partial x_{i}\partial x_{j}\partial y_{i}\partial y_{j}$ should exist and satisfy Dudley's condition in $0$. It actually yields that $f$ has $\mathcal{C}^{2}$ sample paths, and it is not clear wether this is equivalent to $\C^{1,1}$ regularity in this framework. For this reason we state the result with the abstract conditions of Theorem \ref{th:stationary-fields}
.

\begin{theorem}
\label{th:Gauss-stationary}
Let $f=(f(x);x\in \mathbb{R}^{2})$ be a $ \C^{1,1}$ stationary isotropic centred Gaussian  field on $\mathbb{R}^{2}$ with $\E[\Lip(\partial _{i}f)^{p}]<\infty $, for some $p>6$. Let $\lambda \in \mathbb{R}$,  $F=\{x:f(x)\geqslant  \lambda \}$, and let $W\in \W_{2}$ bounded.  Let
$\mu =\E[\partial _{1}f(0)^{2}]$, and $\Phi (\lambda )=\frac{1}{\sqrt{2\pi }}\int_{\lambda }^{\infty  }\exp(-t^{2}/2)dt$. Then
\begin{align}
\label{eq:gauss-exp-volume}
\E [\Vol ^2 (F\cap W)]& =\Vol ^2 (W)\Phi (\lambda ),\\
\label{eq:gauss-exp-perimeter}
\E[ \Per_{\infty } (F\cap W)]&=\Vol ^2 (W)2\frac{\sqrt{\mu }}{\pi }\exp(-\lambda ^{2}/2)+\Per_{\infty }(W)\Phi (\lambda ),\\
\label{eq:gauss-exp-EC}
\E[ \chi (F\cap W)]&=\left(  \Vol ^2 (W)\frac{\mu \lambda} {(2\pi )^{3/2}}+\Per_{\infty }(W)\frac{\sqrt{\mu }}{4\pi }  \right) e^{-\lambda ^{2}/2}+\frac{1}{\sqrt{2\pi }}\Phi (\lambda )\chi (W).
\end{align}
\end{theorem}  
\begin{remark} 
If $W$ is a square, the relation \eqref{eq:gauss-exp-EC} coincides with \cite[(11.7.14)]{AdlTay07}.
\end{remark}

\begin{proof} { 
\eqref{eq:stationary-expectation-Vol} immediately yields \eqref{eq:gauss-exp-volume}. }
To prove \eqref{eq:gauss-exp-EC}, first remark that the stationarity of the field and the fact that it is not constant a.s. entail that $(f(0),\partial _{1}f(0),\partial _{2}f(0))\equlaw (f(x),\partial _{1}f(x),\partial _{2}f(x)),x\in \mathbb{R}^{2}$ is non-degenerate.
Let us show
\begin{align}
\label{eq:intermed-EC-gauss}
\overline{\chi }(F)=\lim_{\varepsilon \to 0}\varepsilon ^{-2}\left[\P \left(
\bar \delta^{\varepsilon }(0,f,\lambda )\right)-\P\left(\bar \delta^{-\varepsilon }(0,-f,-\lambda )\right)
\right]= \frac{\mu \lambda \exp(-\lambda ^{2}/2)}{(2\pi )^{3/2}}.
\end{align}
 
Fix $\varepsilon >0$.
Let $M_{\varepsilon }$ be the $3\times 3$ covariance matrix of $(f(0),f(\varepsilon \u_{1}),f(\varepsilon \u_{2}))$. Since $\text{\rm{Cov}}(f(0),f(\varepsilon \u_{i}))=1-\mu \varepsilon ^{2}/2+O(\varepsilon ^{4}),\text{\rm{Cov}}(f(\varepsilon \u_{1}),f(\varepsilon \u_{2}))=1-\mu \varepsilon ^{2}+O(\varepsilon ^{4})$, straightforward computations show that    
$\det(M_{\varepsilon }) = \varepsilon ^{4}\mu  ^{2}+o(\varepsilon ^{4})$,
 and
\begin{align}
\label{eq:Minverse}
M_{\varepsilon }^{-1}= \frac{1}{ \det(M_{\varepsilon })}\left(  \varepsilon ^{2} W_{\varepsilon } + {\varepsilon ^{4}}  D_{\varepsilon } \right) ,
\end{align}   where the sum of each line and each column of $W_{\varepsilon }$ is $0$, for $\varepsilon >0$, and as $\varepsilon \to 0$
\begin{align*} W_{\varepsilon }\to W:= \mu \left( 
\begin{array}{ccc}
2 & -1 & -1\\
-1 & 1 & 0\\
-1 & 0 & 1\\
\end{array} \right),
D_{\varepsilon }\to D:=\frac{\mu ^{2}}{4} \left( 
\begin{array}{ccc}
-4 & 2 & 2\\
2 & -1 & 1\\
2 & 1 & -1\\ 
\end{array} \right) .
\end{align*} Denote by $\mathbf{1}$ the row vector $(1,1,1)$, and let $\Lambda =\lambda \mathbf{1}$, $Q=\{(t,s,z):t\geqslant 0,s<0,z<0\}$. Denote by $A'$ the transpose of a matrix (or a vector) $A$. We have 
\begin{align*}
\P (\bar \delta^{\varepsilon }(0,f,\lambda ))&=\frac{1}{\sqrt{(2\pi )^{3}\det(M_{\varepsilon })}}\int_{Q+\Lambda }\exp\left( -\frac{1}{2}(t,s,z)'M_{\varepsilon }^{-1}(t,s,z) \right)dtdsdz
\end{align*}and by isotropy and symmetry, for $\lambda \in \mathbb{R},$
\begin{align*}
\P(\bar \delta^{-\varepsilon }(0,-f,-\lambda ))&=\P(\bar \delta^{\varepsilon }(0,-f,-\lambda ))=\P( \bar \delta^{\varepsilon }(0,f,-\lambda) ).
\end{align*} Therefore, \eqref{eq:stationary-expectation-EC} yields that 
 $
\overline{\chi  }(F)=\lim_{\varepsilon \to 0}\varepsilon ^{-2}\left(
\P(\bar \delta^{\varepsilon }(0,f,\lambda) )-\P (\bar \delta^{\varepsilon }(0,f,-\lambda) )
\right).
 $ Let $ X=(t,s,z)\in Q,  Y= {\frac{ \varepsilon  }{\sqrt{\det(M_{\varepsilon })}}}X$.  Since $\Lambda W_{\varepsilon }$ and $W_{\varepsilon }\Lambda $ are $0$, we have  
\begin{align*}
(X+\Lambda )' M_{\varepsilon }^{-1}(X+\Lambda ) 
&=\underbrace{Y'(W_{\varepsilon } +\varepsilon ^{2} D_{\varepsilon })Y}_{=:\gamma _{\varepsilon }(Y)}+\frac{2\varepsilon ^3}{\sqrt{\det(M_{\varepsilon })}}Y'D_{\varepsilon }\Lambda +\frac{\varepsilon ^{4}}{\det(M_{\varepsilon })}\Lambda 'D_{\varepsilon }\Lambda \\
\P(\bar \delta^{\varepsilon }(0,f,\lambda ))&= \frac{\left(
\frac{\sqrt{\det(M_{\varepsilon })}}{\varepsilon }
\right)^{3}\exp\left(
-\lambda ^{2}\frac{\varepsilon ^{4}}{2\det(M_{\varepsilon })}
\mathbf{1}'D_{\varepsilon }\mathbf{1}\right)}{\sqrt{(2\pi )^{3}\det(M_{\varepsilon })}} \int_{Q }\exp\left(-\frac{1}{2}\gamma _{\varepsilon }(Y)-\varepsilon ^{3}\lambda \frac{Y'D_{\varepsilon }\mathbf{1} }{\sqrt{\det(M\varepsilon )}}\right)dY
\end{align*}
and, for some $\theta =\theta (\varepsilon ,Y,\lambda )\in [-\varepsilon ^{3}\det(M_{\varepsilon })^{-1/2},\varepsilon ^{3}\det(M_{\varepsilon })^{-1/2}]$, 
\begin{align*}
\exp\left(-
\frac{\varepsilon ^{3}\lambda Y'D_{\varepsilon }\mathbf{1}}{\sqrt{\det(M_{\varepsilon })}}
\right)-\exp\left(
\frac{\varepsilon ^{3}\lambda Y'D_{\varepsilon } \mathbf{1} }{\sqrt{\det(M_{\varepsilon })}}
\right)=&-2\frac{\varepsilon ^{3}\lambda Y'D_{\varepsilon }\mathbf{1}  }{\sqrt{\det(M_{\varepsilon })}}\exp(\theta \lambda \mathbf{1}  'D_{\varepsilon }Y).
\end{align*}Therefore, as $\varepsilon \to 0,$
$
\varepsilon ^{-2}( \P(\bar \delta^{\varepsilon } (0,f,\lambda ))-\P(\bar \delta^{\varepsilon }(0,f,-\lambda )))$ is equivalent to
\begin{align}\notag& \varepsilon ^{-2}\frac{\exp(-\lambda ^{2}/2)\det(M_{\varepsilon })}{\varepsilon ^{3}\sqrt{(2\pi )^{3}}}\int_{Q}\exp(-\gamma _{\varepsilon }(Y)/2)\frac{-2\varepsilon ^{3}Y'D_{\varepsilon }\lambda \mathbf{1}  }{\sqrt{\det(M_{\varepsilon })}}\exp\left(
 \theta Y'D_{\varepsilon }\lambda \mathbf{1}  
\right)dY\\
\label{eq:expression-EC-I}&\sim \frac{-\exp(-\lambda ^{2}/2)\mu }{\sqrt{2\pi ^{3}}}\int_{Q}\exp(-\gamma _{\varepsilon }(Y)/2)Y'D_{\varepsilon }\lambda \mathbf{1}  \exp(\theta Y'D_{\varepsilon }\lambda \mathbf{1}  )dY.
\end{align}
For $Y=(x,y,z)\in Q$, we have 
\begin{align*}
\frac{Y'WY}{\mu }  =2x^{2}+y^{2}+z^{2}-2xy-2xz=&2x^{2}+y^{2}+z^{2}+2 | xy | +2 | xz | \geqslant \|Y\|^{2}.
\end{align*}   Since $W_{\varepsilon } +\varepsilon ^{2} D_{\varepsilon }\to W$ as $\varepsilon \to 0$, $\gamma _{\varepsilon }(Y)\geqslant \mu   \|Y\|^{2}/2$ for $\varepsilon $ sufficiently small, uniformly in $Y\in Q$. This yields a clear majoring bound and Lebesgue's theorem gives  $\overline{\chi }(F)=-\mu  \exp(-\lambda ^{2}/2)I(2\pi ^{3})^{-1/2}$ with $I=\int_{Q}\exp(-\frac{1}{2}Y'WY)Y'D\Lambda dY=2\lambda J$ where
\begin{align*}
J&=\int_{Q}\exp(-(2t^{2}+s^{2}+z^{2}-2ts-2tz))(s+z)dtdsdz=-1/4
\end{align*} with the change of variables
$ 
u=t-s,
v =t-z, w=t.
 $  The statement \eqref{eq:intermed-EC-gauss} is therefore proved.  The computation of $\overline{\Per_{\infty }}(F)$ is similar and simpler and is omitted here.
\end{proof}

\section{Proofs}

\label{sec:proofs} 

{
\subsection{Proof of Theorem \ref{th:bound-pixel-comp-level-set}}

{\bf (i)} Assume without loss of generality $\lambda =0$ in  the proof. Recall that $\Gamma (F\cap W)$ is the collection of bounded connected components of $F\cap W$. For $0\leqslant k\leqslant d$, {denote by} $\Gamma _{k}(F\cap W)$ the elements of  ${\Gamma (F\cap W)}$ that hit $\partial _{k}W$, and define recursively $\Gamma _{k}^{+}(F;W)=\Gamma _{k}(F\cap W)\setminus \Gamma _{k-1}^{+}(F;W), 1\leqslant k\leqslant d$ (with $\Gamma _{0}^{+}(F;W):=F\cap \text{\rm{corners}}(W)$).

\newcommand{\diam}{\text{\rm{diam}}}
Let $1\leqslant k\leqslant d, C\in \Gamma_{k}^{+}(F;W) $,   $C'$ arbitrarily chosen in $ \Gamma (C\cap \partial _{k}W)$. Since $C'$ does not touch $\partial _{k-1}W$, it is included in the relative interior of $\partial _{k}W$ within the affine $k$-dimensional tangent space to ${\partial}W$ that contains $C'$, hence it is contained in some facet $G$. Let $I\in \mathcal{I}_{ k}$ such that for $x\in G,I_{x}(W)=I$.  
Let $x_{C}\in \text{\rm{cl}}(C')$ be such that $f(x_{C})=\sup_{C'}f$. Since the  $k$-dimensional gradient $\nabla (f_{ | G})$ does not vanish on $\partial C'$,  $f(x_{C})>0$ and
the Lagrange multipliers Theorem yields that  $\partial _{i }f(x_{C})=0$ for $i\in I$.
 Call $r_{C}$ the maximal radius such that $B_{C}:=(B(x_{C},r_{C})\cap G)\subset C'$.   Since $  B _{C}$ touches $\partial F$,   $f$ has a zero on $B_{C}$. It follows that $ | f(x) | \leqslant 2\Lip(f)r_{C}$ and $ | \partial _{i}f(x) | \leqslant \Lip(\partial _{i}f )r_{C}$ for $x\in B_{C},i\in I $. 
 Define  $$M(x)=\max \left(
\frac{| f(x) |}{ 2\Lip(f )}  ,\frac{ | \partial _{i}f(x) |}{\Lip(\partial _{i}f )} ,i\in I
\right)\in \mathbb{R}_{+} ,x\in \partial _{k}W.$$ Since $M(x)\leqslant r_{C}$ on $B_{C}$, we have
\begin{align*}
1=&\frac{1}{\mathcal{H}_{d}^{k}(B_{C})}\int_{B_{C}}\mathbf{1}_{\{M(x)\leqslant r_{C} \}}\mathcal{H}_{d}^{k}(dx)=  {\kappa _{k}^{-1}r_{C}^{-k}} \int_{B_{C}}\mathbf{1}_{\{r_{C} ^{-1}\leqslant M(x)^{-1}\}}\mathcal{H}_{d}^{k}(dx)\\
&\leqslant \kappa _{k}^{-1}\int_{B_{C}}M(x)^{- k}\mathcal{H}_{d}^{k}(dx).
\end{align*}Since the $B_{C}$ are pairwise disjoint, summing over all the $C\in \Gamma _{k}^{+}(F;W)$ and  $k\in \{1,\dots ,d\}$ gives
\begin{align*}
\#\Gamma (F\cap W)\leqslant  \sum_{k=0}^{d}\Gamma _{k}^{+}(F;W) \leqslant \sum_{k=0}^{d}\sum_{C\in \Gamma _{k}^{+}(F;W)}\kappa _{k}^{-1}\int_{B_{C}}M(x)^{-k}\mathcal{H}_{d}^{k}(dx) \leqslant \sum_{k=0}^{d}\kappa _{k}^{-1}\int_{\partial _{k}W}M(x)^{-k}\mathcal{H}_{d}^{k}(dx).
\end{align*}  
with $M (x)^{-0}= 1$ by convention on $\text{\rm{corners}}(W)$. The result follows by noticing that 
\begin{align*}
M(x)\geqslant  \min\left(
\frac{1}{2\Lip(f)},\frac{1}{\Lip(\partial _{i}f)},i\in I
\right){\max( | f(x) | , | \partial _{i}f(x) | ,i\in I) }\geqslant \frac{\max( | f(x) | , | \partial _{i}f(x) | ,i\in I)}{2\max( \Lip(f),\Lip(\partial _{i}f),i\in I)}.
\end{align*} 
{ {\bf (ii)}
Theorem 2.12 in the companion paper \cite{LacEC1}, in the context $d=2$, features a bound on $\chi ((F\cap W)^{\varepsilon })$ in terms of the number of occurrences of local configurations called \textit{entanglement points of $F$}. Roughly speaking, an entanglement point 
occurs when two close points of $F$ are connected  by a tight path in $F$. 
As a consequence, if $F$ is sampled with an insufficiently high resolution in this region, the connecting path is not detected, and $F$ looks locally disconnected.
For formal definitions, for $x,y\in \varepsilon \mathbb{Z} ^{2}$ at distance $\varepsilon $, introduce $\p_{x,y}$ the closed  square with side-length $\varepsilon $ such that   $x$ and $y$ are the midpoints of two opposite sides. Let $\p_{x,y}'=\partial \p_{x,y}\setminus \{x,y\}$, which has two connected components. Then $\{x,y\}$ is an \textit{entanglement pair of points} of $F$ if $x,y\notin F$ and   $(\p'_{x,y}\cup F)\cap \p_{x,y}$ is connected.  
We call $\N_{\varepsilon }(F)$ the family of such pairs of points.  See Figure \ref{fig:entang} for an example.

\pdfoutput=1
\begin{figure} [h!]
\centering
\caption{\label{fig:entang}\textit{Entanglement point:} In this example, $\{x,y\}\in \N_{\varepsilon }(F)$ because the two connected components of $\p'_{x,y}$, in  grey, are connected through $\gamma \subseteq (F\cap \p_{x,y})$. We have $\{x,y\}\notin  \N_{\varepsilon }(F')$.}
\includegraphics[scale=.2]{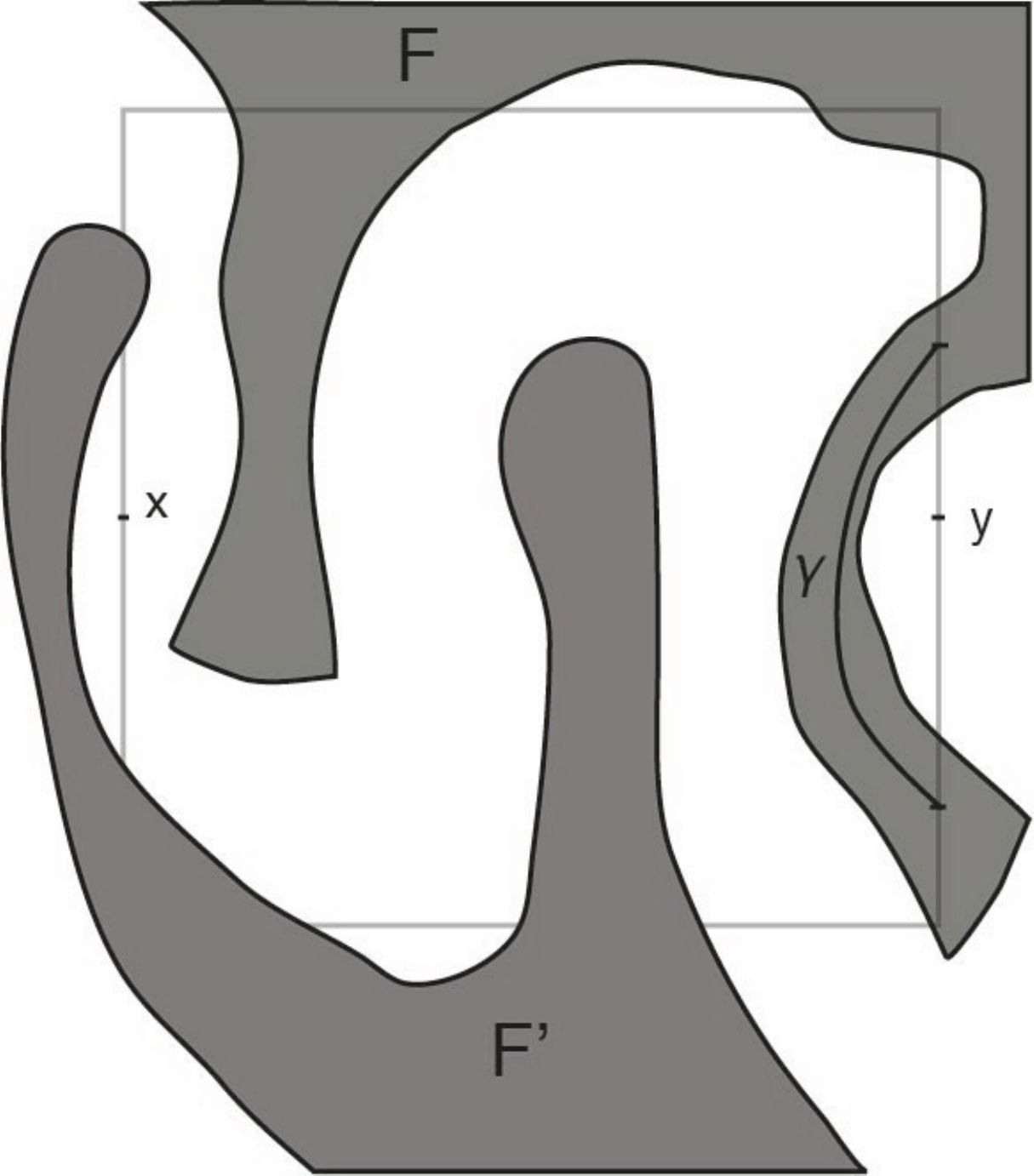}
\end{figure}
Let $\llparenthesis x,y \rrparenthesis=\varepsilon \mathbb{Z} ^{2}\cap [x,y]\setminus \{x,y\}, $ for $x,y\in \varepsilon \mathbb{Z} ^{2} $. For $A\subset \mathbb{R}^{2}$, note $A^{\oplus \varepsilon }=\{x\in \mathbb{R}^{d}:d(x,A)\leqslant \varepsilon \}$.
 To account for boundary effects, we also consider grid points $x,y \in   \varepsilon \mathbb{Z} ^{2}\cap  W\cap F$, on the same line or column of $\varepsilon \mathbb{Z} ^{2}$,  such that \begin{itemize}
\item $x,y$ are within distance $\varepsilon $ from  one of the  edges of $W$ (the same edge for $x$ and $y$)  
\item $\llparenthesis x,y\rrparenthesis \neq \emptyset $
\item $\llparenthesis x,y\rrparenthesis \subseteq \varepsilon \mathbb{Z} ^{2}\cap F^{c}\cap F^{\oplus \varepsilon }$.
\end{itemize}  The family of such pairs of points $\{x,y\}$ is denoted by $\N_{\varepsilon }'(F;W)$ .
 It is proved in \cite[Theorem 2.12]{LacEC1} that  \begin{align}
\label{eq:bound-Gamma-eps-W}
  \#\Gamma ((F\cap W)^{\varepsilon }) &  \leqslant   2\#(\N_{\varepsilon }(F)\cap  W^{\oplus \varepsilon })+2\#\N_{\varepsilon }'(F,W) + \#\Gamma (F\cap W)+2\#\corners(W ) .
  \end{align}
It therefore only remains to bound $\#(\N_{\varepsilon }(F)\cap W^{\oplus \varepsilon })$ and $\#\N'_{\varepsilon }(F,W)$ to achieve \eqref{eq:bound-gamma-level-set-pixel}. For $m\geqslant 1$ and a function $g:A\subseteq \mathbb{R}^{m}\to \mathbb{R} $, introduce the continuity modulus
\begin{align*}
\omega (g ,A)=\sup_{x\neq y\in A  }{|g(x)-g(y)|} .
\end{align*}  
 The bound will follow from the following lemma.  \begin{lemma}
\label{lm:bound-Ne-levelset}\textbf{(i)}
For $\{x,y\}\in \N_{\varepsilon }(F )$, we have for some $i\in \{1,2\}$, and $i'=i+1 \text{\rm{ mod }}2$,
\begin{align*}
 | f(x) |&\leqslant \omega (f,[x,y])\leqslant \Lip(f)\varepsilon \\
  | \partial_i f(x) |&\leqslant \omega (\partial_i f,[x,y])\leqslant \Lip(\partial_i f)\varepsilon \\
   | \partial_{i'} f(x) |&\leqslant 2\omega (\partial_i f,\p_{x,y})+\omega (\partial_{i'} f,\p_{x,y}) \leqslant \sqrt{2}\varepsilon (2\Lip(\partial_i f)+\Lip(\partial_{i'} f)),
\end{align*}and idem for $y$.

\textbf{(ii)} For $x,y\in \N_{\varepsilon }'(F,W)$, there is $z=z(x,y)\in \llparenthesis x,y\rrparenthesis$, $i\in \{1,2\}$, such that 
\begin{align*}
 | f(z) |&\leqslant \Lip(f)\varepsilon\\ 
  | \partial _{i}f(z) |& \leqslant  \Lip(\partial _{i}f)\varepsilon  .
\end{align*}\end{lemma} 
The lemma is proved later for convenience.
To obtain the integral upper bounds from \eqref{eq:bound-gamma-level-set-pixel}, note that there is $c>0$ such that for $\varepsilon >0$ sufficiently small, for every  $x,y\in W$, neighbours in $\varepsilon \mathbb{Z} ^{2}$, $\Vol ^2 ((B(x,\varepsilon )\cup B(y,\varepsilon) )\cap W)\geqslant \varepsilon ^{2}/c$. Define the possibly infinite quantity, for $z\in W,$
\begin{align*}M(z)=&\max\left(
\frac{ | f(z) |}{   2\Lip(f)  },  \frac{ | \partial _{i}f(z) | }{  2\Lip(\partial _{i}f)} ,\frac{ | \partial _{i'}f(z) | }{ 2\sqrt{2}(\Lip(\partial_1 f)+\Lip(\partial_2 f))}
\right).
\end{align*}
Lemma 
\ref{lm:bound-Ne-levelset} yields that for $\{x,y\}\in \N_{\varepsilon }(F)$ and $z\in B(x,\varepsilon )\cup B(y,\varepsilon )$, $M(z)\leqslant \varepsilon $.
  Then
\begin{align*}
\#\N_{\varepsilon }(F)&
\leqslant 
\sum_{\{x,y\}\in \N_{\varepsilon }(F)}
 \mathbf{1}_{\{  \forall z\in B(x,\varepsilon )\cup B(y,\varepsilon )\cap W,
M(z)\leqslant \varepsilon  
  \}}\\
&\leqslant
 \sum_{\{x,y\}\in \N_{\varepsilon }(F)} 
 \int\limits_{(B(x,\varepsilon )\cup B(y,\varepsilon ))\cap W} 
 {c}{  \varepsilon ^{-2}} \mathbf{1}_{\{\varepsilon ^{-1}\leqslant 
  M(z)^{-1}  
      \}}dz\\
&\leqslant 
4c \int_{  W}M(z)^{-2}dz\leqslant c'I_{2}(f;W),\\
\end{align*} for some $c'>0$, because for every $z\in W$ there are at most 4 couples $\{x,y\}\in \N_{\varepsilon }(F)$ such that $z\in B(x,\varepsilon )\cup B(y,\varepsilon )$.  

Now, given $w\in \partial W$, there can be at most $3$ pairs $\{x,y\}\in \N_{\varepsilon }'(F)$ such that $w$ is on the closest edge of $W$ parallel to $[x,y]$ and $z=z(x,y)$ (defined in Lemma \ref{lm:bound-Ne-levelset})  is within distance $3\varepsilon $ from $w$, and in this case $ | f(w) | \leqslant 4\Lip(f)\varepsilon $ and $ | \partial _{i}f(w) | \leqslant 4\Lip(\partial _{i}f)\varepsilon $ for some $i\in \{1,2\}$. We have $\mathcal{H}^{1}_{2}(B(z,3\varepsilon )\cap \partial W)\geqslant \varepsilon $, because $z$ is within distance $2\varepsilon $ from a segment of $\partial W$ parallel to $[x,y]$. It follows that, with $M_{i}(w)=\max( | f(w) | /4\Lip(f) , | \partial _{i}f(w) | /4\Lip(\partial _{i}f))$
\begin{align*}
\#\N_{\varepsilon }'(F,W)
&\leqslant 
\sum_{\{x,y\}\in \N_{\varepsilon }'(F,W)}\sum_{i=1}^{2}
\mathbf{1}_{\{ \forall w\in B(z,3\varepsilon )\cap \partial W,
M_{i}(w)\leqslant \varepsilon     \}}\\
&\leqslant
\sum_{i=1}^{2} \sum_{\{x,y\}\in \N_{\varepsilon }'(F)}
\frac{1}{\varepsilon }\int_{\partial W\cap B(z,3\varepsilon )}
\mathbf{1}_{\{ 
M_{i}(w)^{-1}\geqslant \varepsilon ^{-1}\}}\mathcal{H}_{2}^{1}(dw)\\
&\leqslant 
\sum_{i=1}^{2}
 {3} \int_{\partial W}
M_{i}(w) ^{-1}
    \mathcal{H}_{2}^{1}(dw)\leqslant 24I_{1}(f;W).
\end{align*} 
}

\begin{proof}[Proof of Lemma \ref{lm:bound-Ne-levelset}] 
 For $x\in \mathbb{R}^{2}$,  denote by $(x_{[1]},x_{[2]})$  its coordinates in the canonical basis, not to be  mistaken with a pair of vector of $\mathbb{R}^{2}$, denoted by $(x_{1},x_{2})$. If $\varphi $ is {a mapping} with values in $\mathbb{R}^{2}$, denote its coordinates by $(\varphi (\cdot )_{[1]} ,\varphi (\cdot )_{[2]})$.

\textbf{(i)}
Let $x,y\in \N_{\varepsilon }(F)$.
The definition of $\N_{\varepsilon }(F)$ yields a connected path $\gamma \subseteq (F\cap \p_{x,y})$ going through some $z\in [x,y]$ and connecting the two connected components of $\p'_{x,y}$. Since $f(x)\geqslant 0$ and $f(z)\leqslant 0$,   there is a point $z'$ of $[x,y] $ satisfying $f(z')=0 $, hence $ | f(x) | \leqslant \omega (f,[x,y])$. Note for later that for $t\in \p_{x,y}$  $ | f(t)  | \leqslant\omega (f,\p_{x,y})  \leqslant \Lip(f)\sqrt{2}\varepsilon  $.

 We assume without loss of generality  that $[x,y]$ is horizontal. Let $[z',z'']$ be the (also horizontal) connected component of $F\cap [x,y]$ containing $z$. After choosing a direction on $[x,y]$, $z'$ and $z''$ are entry and exit points for $F$, and their normal vectors $\n_{F}(z'),\n_{F}(z'')$  point towards the outside of $F$. Therefore they satisfy  $\n_{F}(z')_{[1]}\n_{F}(z'')_{[1]}\leqslant  0$, and so $\partial_1 f(z')\partial_1 f(z'')\leqslant  0$. This gives us by continuity the existence of a point $w\in [x,y]$ such that $ 0=\partial_1 f(w)$, whence $ |  \partial _{1}f(x) | \leqslant \omega (\partial _{1}f,[x,y])$. Note for later that $ | \partial_1 f(t)| \leqslant \omega (\partial_1 f,\p_{x,y})   $ on $\p_{x,y}$. If $[x,y]$ is vertical, $\partial_2 f$ verifies the inequality instead. Let us keep assuming that $[x,y]$ is horizontal for the sequel of the proof. 
 
  We claim that $ | \partial_2 f(x)  | \leqslant  2\omega (\partial_1 f, \p_{x,y})+\omega (\partial_2 f, \p_{x,y} )  $, and consider two cases to prove it.

\begin{itemize}
\item \underline{First case}   $\partial_2 f(z')\partial_2 f(z'')\leqslant  0$, and by continuity we have $w\in [x,y]$ such that $0= \partial_2 f(w)  $, whence $ | \partial_2 f(\cdot ) |\leqslant  \omega (\partial_2 f, \p_{x,y}) $ on the whole pixel $\p_{x,y}$. The desired inequality follows.
\item \underline{Second case}   $\partial_2 f(z')>0,\partial_2 f(z'')>0$ (equivalent treatment if they are both $<0$).  Assume for instance that $z'$ is the leftmost point, and that $ | \partial_2 f(x) | >2\omega (\partial_1 f ,\p_{x,y})+\omega (\partial_2 f ,\p_{x,y} ) $, otherwise the claim is proved. It implies in particular that $ | \partial_2 f(\cdot ) |>2 \omega (\partial_1 f, \p_{x,y}) $  on the whole pixel $\p_{x,y}$.  
Since  $ | \partial _{1}f(\cdot ) | \leqslant \omega (\partial _{1}f,\p_{x,y})$ on $\p_{x,y}$, the implicit function theorem yields a  unique function $\varphi$ (resp. $\psi $)$  :[z'_{[1]},z_{[1]}'']\to \mathbb{R}$ such that $\varphi (z_{[1]}')=z_{[2]}'$ (resp. $\psi (z_{[1]}'')=z_{2}''=z_{[2]}'$), $ | \varphi ' | \leqslant 1/2$, (resp. $| \psi ' | \leqslant 1/2)$,  $\varphi ([z_{[1]}',z_{[1]}''])\subset (z'_{[2]}+(-\varepsilon /2,\varepsilon /2))$, (resp. $\psi ([z_{[1]}',z_{[1]}''])\subset (z'_{[2]}+(-\varepsilon /2,\varepsilon /2))$) and the graph of $\varphi $ (resp. $\psi $) coincides with $\partial {F}\cap ([z_{[1]}',z_{[1]}'']\times (z'_{[2]}+[-\varepsilon /2,\varepsilon /2])$. In particular, $\varphi =\psi $, and its graph cannot touch the upper half of $\partial \p_{x,y}$. Applying this to every maximal segment $[z',z'']\subset (F\cap [x,y])$, we see that every connected component of $F$ touching $[x,y]$, and hence $\gamma $, cannot meet the upper half of $\p_{x,y}$. In particular, it contradicts the definition of $\N_{\varepsilon }(F)$, whence indeed the assumption is proved by contradiction.
\end{itemize}
\textbf{(ii)}Let now $\{x,y\}$ be an element of $\N_{\varepsilon }'(f,W)$.  We know that $\llparenthesis x,y\rrparenthesis \cap F^{c}\neq \emptyset $. Let $[z',z'']\subset [x,y]$ be a connected component  of $F^{c}\cap [x,y]$. If $[z',z'']$ is, say, horizontal, since $\n_{F}(\cdot )_{[1]}$ changes sign between $z'$ and $z''$, so does $\partial_1 f$, and by continuity there is $w\in [z',z'']$ where $\partial_1 f(w)=0$. Calling $z$ the closest point from $w$ in $\llparenthesis x,y\rrparenthesis$, $\|z-w\|\leqslant \varepsilon $, and by definition of $\N_{\varepsilon }'(F,W)$, $z$ is also at distance $\varepsilon $ from $\partial F=\{f=0\}$. It follows that the result holds with $i=1:$ $ | \partial_1 f(z) |\leqslant \Lip(\partial_1 f)\varepsilon ,  | f(z) |\leqslant \Lip(f)\varepsilon  $. A similar argument holds for $i=2$ if $[z',z'']$ is vertical.
 \end{proof}

\subsection{Proof of Theorem \ref{th:conditions-random-excursion}}
Assume without loss of generality  $\lambda =0$. 
Let us prove that $f$ is a.s. regular within $W$ at level $0.$
For $0\leqslant k\leqslant d,I\in \mathcal{I}_{ k}$,  define 
\begin{align*}
 \theta _{k,I}=\{x\in  \partial _{k}W:f(x)=0,\partial _{i}f(x)=0,i\in I\}.
\end{align*}We must prove that $\theta _{k,I}=\emptyset $ a.s.. Define $M(y)=\max( | f(y) | /\Lip(f), | \partial _{i}f(y) | /\Lip(\partial _{i}f),i\in I)$.
For $x\in \theta _{k,I},y\in B(x,\varepsilon ) $, we have $M(y)\leqslant \varepsilon $. 
Since $W$ is a polyrectangle, there is $c_{W}>0$ such that the following holds for $\varepsilon >0$ sufficiently small: there is a partition $ \{C_{i}^{\varepsilon };1\leqslant i\leqslant n_{\varepsilon }\}$ of $\partial_{k}W$ such that for each $i,C_{i}^{\varepsilon }\subset B(x_{i}^{\varepsilon },\varepsilon/2 )$ for some $x_{i}^{\varepsilon }\in \partial _{k}W$, and $\mathcal{H}^{k}_{d}(C_{i}^{\varepsilon })\geqslant \varepsilon ^{k}/c_{W}$.  Then for any $\eta \geqslant  0,$
\begin{align*}
\#\theta _{k,I}=&\lim_{\varepsilon \to 0}\#\{1\leqslant i\leqslant n_{\varepsilon }:\theta _{k,I}\cap C_{i}^{\varepsilon }\neq \emptyset \}\\
=&\lim_{\varepsilon \to 0}\sum_{i=1}^{n_{\varepsilon }}\int_{C_{i}^{\varepsilon }}\mathbf{1}_{\{M(y)\leqslant \varepsilon \}}\frac{\mathcal{H}^{k}_{d}(dy)}{\mathcal{H}^{k}_{d}(C_{i}^{\varepsilon })}\\
\leqslant& \liminf_{\varepsilon \to 0}\frac{c_{W}}{\varepsilon ^{k}}\int_{\partial _{k}W}   \mathbf{1}_{\{M(y)\leqslant \varepsilon \}}\mathcal{H}^{k}_{d}(dy)\\
\leqslant &c_{W} \liminf_{\varepsilon \to 0}\varepsilon ^{\eta }\int_{ \partial _{k}W}M(y)^{-(k+\eta )}\mathcal{H}^{k}_{d}(dy).
\end{align*}
For any numbers $a_{1},\dots ,a_{q},b_{1},\dots ,b_{q}>0$, we have 
\begin{align*}
\max_{i}(a_{i}/b_{i})\geqslant \max_{i}(a_{i}\min_{j}(1/b_{j}))=\min_{j}(1/b_{j})\max_{i}(a_{i})=\frac{\max_{i}(a_{i})}{\max_{i}(b_{i})}.
\end{align*} 
Hence, with $L:=\max(\Lip(f),\Lip(\partial _{i}f),i\in I),m_{y}:=\max( | f(y) | , | \partial _{i}f(y) | ,i\in I),$ $M(y)\geqslant m_{y}/L$.
Fatou's lemma yields, 
\begin{align*}
\E[\#\theta _{k,I}]\leqslant &c_{W}\liminf_{\varepsilon \to 0}\varepsilon ^{\eta }\int_{\partial _{k}W}\E\left[
\frac{L^{k+\eta }}{m_{y}^{k+\eta }}
\right]\mathcal{H}^{k}_{d}(dy).
\end{align*}
Since $p>\frac{\alpha d}{\alpha -1}\geqslant \frac{\alpha k}{\alpha -1}$, we can choose $q\in ((1-k/p)^{-1},\alpha ) $. It satisfies $q>1$  and $q'k<p$, with $q'=(1-q^{-1})^{-1}$. In particular, if $\eta $ is chosen $>0$ sufficiently small, $C:=\E [L^{q'(k+\eta )}]<\infty $ for $1\leqslant k\leqslant d$.
Using H\"older's inequality,   $\mathbb{E}(\theta _{k,I})=0$ follows from the following bound, uniform in $y$:
\begin{align*}
\E [m_{y}^{-q(k+\eta )}]\leqslant &  \int_{\mathbb{R}_{+}}\P(m_{y}\leqslant t^{-1/(q(k+\eta ))})dt\\
\leqslant & \int_{\mathbb{R}_{+}} \P ( | f(y) | \leqslant t^{-1/(q(k+\eta ))}, | \partial _{i}f(y) | \leqslant t^{-1/(q(k+\eta ))},i\in I)dt\\
\leqslant &c\int_{\mathbb{R}_{+}}1\wedge t^{-\frac{\alpha k}{q(k+\eta )}}dt  .
\end{align*}
Assume without loss of generality that $\eta >0$ is chosen so that $q(k+\eta )<\alpha k $, so that indeed for all $1\leqslant k\leqslant d,$ the previous bound is finite (uniformly in $y$).

The proof that the $I_{k}(f;W)$ have finite expectation for $1\leqslant k\leqslant d$ is similar. We have 
\begin{align*}
\E [I_{k}(f;\lambda )]\leqslant (\E [L^{kq'}])^{\frac{1}{q'}}\int_{\partial _{k}W}(\E [m_{y}^{-kq}])^{\frac{1}{q}}\mathcal{H}^{k}_{d} (dy).
\end{align*}The finiteness follows by the exact same computation as before, with $\eta =0.$

Therefore, if $d=2$, using \eqref{eq:deterministic-expression-chi}, $\chi (F\cap W)=\lim_{\varepsilon \to 0}\chi ((F\cap W)^{\varepsilon })$  holds a.s. According to Theorem \ref{th:bound-pixel-comp-level-set}, the finiteness of $\E[ I_{1}(f;W)] $ and $\E[ I_{2}(f;W)] $  enable us to  apply Lebesgue's theorem and show that this limit can be passed to expectations, yielding \eqref{eq:expect-EC-excursion}-\eqref{eq:expect-EC-covariogram}. 

\subsection{Proof of Theorem \ref{th:stationary-fields} }
\label{sec:proof-4-2}
According to the previous proof, the quantities $I_{k}(f;\lambda ),0\leqslant k\leqslant 2,$ have finite expectation on any bounded $W\in \W_{2}$. Hence according to the proof of Theorem \ref{th:bound-pixel-comp-level-set},  the quantities $\E [\sup_{0<\varepsilon \leqslant 1}\#(\mathscr{N}_{\varepsilon }(F)\cap W)],$ $ \E[ \sup_{0<\varepsilon \leqslant 1}\#(\mathscr{N}_{\varepsilon }(F^{c})\cap W)]$, $\E[ \sup_{0<\varepsilon \leqslant 1}\#(\mathscr{N}_{\varepsilon }'(F,W))]$,  $\E [\sup_{0<\varepsilon \leqslant 1}\#(\mathscr{N}_{\varepsilon }'(F^{c},W))]$, $\E[\# \Gamma (F\cap W)]$, $\E [\#\Gamma (F^{c}\cap W)]$ are finite. Therefore Theorem \ref{th:stationary-fields} is a consequence of \cite[Proposition 3.1]{LacEC1}.
     
 \bibliographystyle{plain}
    \bibliography{../../../../bicovariograms}
    
 \end{document}